\documentclass[leqno,11pt, a4]{amsart}
\usepackage[utf8]{inputenc}
\usepackage{amsthm}
\usepackage{amsmath}
\usepackage{enumitem}
\usepackage{amsfonts}
\usepackage{amssymb}
\usepackage{graphicx}
\usepackage{xcolor}
\usepackage{pdfsync}
\newtheorem{theorem}{Theorem}[section]

\newtheorem{corollary}{Corollary}[theorem]
\newtheorem{lemma}[theorem]{Lemma}

\newcounter{yuppo}
\newcommand{\Keler} {K\"{a}hler }

\newcommand{\End}{\operatorname{End}}

\newcommand{\cds}{\cdots}
\newcommand{\cd}{\cdot}

\renewcommand{\phi}{\varphi}

\newcommand{\ra}{\rightarrow}
\newcommand{\lra}{\longrightarrow}
\newcommand{\C}{\mathbb{C}}
\newcommand{\R}{\mathbb{R}}

\newcommand{\Gl}{\operatorname{Gl}}

\newcommand{\zero}{\bigg \vert _{t=0} }

\newcommand{\liu}{\mathfrak{u}}

\newcommand{\lia}{\mathfrak{a}}

\newcommand{\lieg}{\mathfrak{g}}

\newcommand{\liep}{\mathfrak{p}}

\newcommand{\liea}{\mathfrak{a}}
\newcommand{\noparty}[1]{}%{\colorbox{Apricot}{#1}}

\newcommand{\mup}{\mu_\liep}
\newcommand{\mupb}{\mu_\liep^\beta}
\title{Common singularities  of commuting vector fields }
\author{Leonardo Biliotti}
\address{Dipartimento di Scienze Matematiche, Fisiche e Informatiche \\
          Universit\`a di Parma (Italy)}
\email{leonardo.biliotti@unipr.it}
\author{Oluwagbenga Joshua Windare}
\address{Dipartimento di Scienze Matematiche, Fisiche e Informatiche \\
          Universit\`a di Parma (Italy)}
\email{oluwagbengajoshua.windare@unipr.it}
\thanks{The first author was partially supported by PRIN  2017
   ``Real and Complex Manifolds: Topology, Geometry and holomorphic dynamics '' and GNSAGA INdAM. The second author was supported by GNSAGA of INdAM}
\begin{document}
\maketitle
\begin{abstract}
\noindent
We study the singularities of commuting vectors fields of a real submanifold of a K\"ahler manifold $Z$. 
\end{abstract}
%\newpage
%\renewcommand{\contentsname}{Table of Contents}
%\tableofcontents
\section{Introduction}
\pagenumbering{arabic}
Let $(Z,\omega)$ be a connected  \Keler manifold with an holomorphic action of a complex reductive group $U^\C$, where $U^\C$ is the complexification of a compact connected Lie group $U$ with Lie algebra $\mathfrak{u}$. We also assume $\omega$ is $U$-invariant and that there is a $U$-equivariant momentum map $\mu : Z \to \mathfrak{u}^*.$ By definition, for any $\xi \in \mathfrak{u}$ and $z\in Z,$ $d\mu^\xi = i_{\xi_Z}\omega,$ where $\mu^\xi(z) := \mu(z)( \xi )$ and $\xi_Z$ denotes the fundamental vector field induced on $Z$ by the action of $U,$ i.e.,
$$\xi_Z(z) := \frac{d}{dt}\zero \exp(t\xi)z
$$
(see, for example, \cite{Kirwan} for more details on the momentum map).
Since $U$ is compact we may identify $\mathfrak{u}  \cong \mathfrak{u}^*$ by an
$\mathrm{Ad}(U)$-invariant scalar product on $\mathfrak{u}$. Hence, we consider a momentum map as a $\mathfrak{u}$-valued map, i.e., $\mu : Z \to \mathfrak{u}$. Recently, the momentum map has been generalized to the following settings \cite{PG,heinzner-schwarz-stoetzel}.

We say that a subgroup $G$ of $U^\C$ is compatible if $G$ is closed and the Cartan decomposition $U^\C=U\exp(\text{i} \liu)$ induces a Cartan decomposition of $G$. This means that the map $K\times \mathfrak{p} \to G,$ $(k,\beta) \mapsto k\text{exp}(\beta)$
is a diffeomorphism where $K := G\cap U$ and $\mathfrak{p} := \mathfrak{g}\cap \text{i}\mathfrak{u};$ $\mathfrak{g}$ is the Lie algebra of $G.$
%The Lie algebra $\mathfrak{u}^\C$ of $U^\C$ is the direct sum $\mathfrak{u}\oplus \text{i}\mathfrak{u}.$ It follows that $G$ is compatible with the
%Cartan decomposition $U^\C = U\text{exp}(\text{i}\mathfrak{u})$. 
In particular $K$ is a maximal compact subgroup of $G$ with Lie algebra $\mathfrak{k}$ and that
$\mathfrak{g} = \mathfrak{k}\oplus \mathfrak{p}$. %The adjoint action of $K$ preserves $\mathfrak{p}.$

The inclusion i$\mathfrak{p}\hookrightarrow \mathfrak{u}$ induces by restriction, a $K$-equivariant map $\mu_{\text{i}\mathfrak{p}} : Z \to (\text{i}\mathfrak{p})^*$ by composing the momentum map $\mu$ with the restriction map $\liu^* \to (\text{i} \liep)^*$. Using a Ad$(U)$-invariant scalar product on $\text{i}\liu$ requiring multiplication by $\text{i}$ to be an isometry between $\liu$ and $\text{i}\liu$, $\mu_{\text{i}\mathfrak{p}}$ can be viewed as the orthogonal projection of $\text{i}\mu(z)$ onto $\liep$ given as $\mu_\liep : Z \ra \liep$. Let $\mu_\liep^\beta(z) := \langle\mu_\liep(z), \beta\rangle = \langle \text{i}\mu(z), \beta\rangle = \langle \mu(z), -\text{i}\beta\rangle = \mu^{-\text{i}\beta}(z)$
% The map $\beta \mapsto -i\beta $ is an injection $\liep
% \hookrightarrow \liu$.
for any $\beta \in \mathfrak{p}$ and $ z\in Z.$  Then grad$\, \mu_\mathfrak{p}^\beta = \beta_Z$ where grad is computed with respect to the Riemannian metric induced by the \Keler structure. The map $\mu_\mathfrak{p}$ is called the gradient map associated with $\mu$. In this paper, a $G$-invariant compact connected locally closed real submanifold $X$ of $Z$ is fixed and the restriction of $\mu_\mathfrak{p}$ to $X$ is denoted by $\mu_\mathfrak{p}$. Then $\mup:X \to \liep$ is a $K$-equivariant map such that $
\text{grad}\mu_\mathfrak{p}^\beta = \beta_X,
$ where the gradient is computed with respect to the induced Riemannian metric on $X$ denoted by $(\cdot,\cdot)$. By the Linearization Theorem  \cite{heinzner-schwarz-stoetzel,sjamaar}, $\mu_\liep^\beta$ is a Morse-Bott function \cite{LA,heinzner-schwarz-stoetzel}  and the  limit 
$$
\phi_\infty^\beta (x):=\lim_{t\to +\infty} \exp(t\beta )x
$$
exists and it belongs to $X^\beta :=\{z\in X:\, \beta_X (z)=0\}$ for any $x\in X$.  The Linearization Theorem \cite{heinzner-schwarz-stoetzel,sjamaar} also proves that any connected component of $X^\beta$ is an embedded submanifold, see for instance \cite{LA,heinzner-schwarz-stoetzel}.  

Let $c_1 > \cds > c_k$ be the critical values of $\mup^\beta.$ Let $C_1,\ldots, C_k$ be the connected components of $X^\beta$ and $W_i:=\{x\in X:\, \lim_{t\to +\infty} \exp(t\beta)x \in C_i\}$. Then $\mup^\beta (C_i)=c_i$ and applying again the Linearization Theorem \cite{heinzner-schwarz-stoetzel,sjamaar}, the submanifold $C_i$ is a connected component of  $(\mup^\beta)^{-1} (c_i)$.  One of the most important Theorem of Morse theory proves that $W_i$ is an embedded submanifold, which is called \emph{unstable manifold} of the critical submanifold $C_i$, and $\phi_\infty^\beta : W_i \to C_i$ is smooth \cite{Bott}.

Let $T$ be a torus of $U$. This means that $T$ is a connected compact Abelian subgroup of $U$ \cite{ad}. By a Theorem of Koszul, \cite{DK}, the connected components of $Z^T:=\{x\in Z:\, T\cdot x=x\}$ are embedded K\"ahler submanifolds of $Z$. Let $\mathfrak t$ be the Lie algebra of $T$. It is well-known that the set
$$
\left\{\beta\in \mathfrak t:\, \overline{\exp(\R \beta)}=T \right\},
$$
contains a dense subset \cite{ad}. Hence,
\begin{equation}\label{complex}
Z^T=Z^{T^\C}=\left\{p\in Z:\, \beta_Z(p)=0\right\},
\end{equation}
for some $\beta \in \mathfrak t$. This means $Z^T$ is the set of the singularities of the vector field $\beta_Z$. Moreover,  $Z^T$ is the image of the gradient flow $\phi_\infty^\beta$ defined by $\mu^\beta$.

In this paper, we  investigate the fixed point set of the action of an  Abelian compatible subgroup of $U^\C$ acting on a real submanifold of $Z$. Let $\lia \subset \liep$ be an Abelian subalgebra and $A=\exp(\lia)$. Then the $A$-gradient map on $X$ is given by  $\mu_{\lia}=\pi_\lia \circ \mup$, where  $\pi_\lia:\liep \lra \lia$ denotes the orthogonal projection of $\liep$ onto $\lia$. Since $A$ is Abelian, then  by Lemma \ref{compatible},  for any $p\in X,$ the stabilizer $A_p :=\{a\in A:\, ap=p\}=\exp (\lia_p)$, where $\lia_p$ is the Lie algebra of $A_p$.  Therefore $X^A=\{p\in X:\, A\cdot p=p \}=\{ p\in X:\, \beta_X (p)=0,\, \forall \beta \in \lia\}$. Hence, if $\alpha_1,\ldots,\alpha_n$ is a basis of $\lia,$ then $X^A$ is the set of the common singularities of the commuting vector fields $(\alpha_1)_X, \ldots, (\alpha_n)_X$.
Our first main result is the following
\begin{theorem}
The set $\left\{\beta \in \lia:\, X^\beta=X^A \right\}$ is dense in $\lia$.
\end{theorem}
Hence  $X^A$ is the set of the singularities  of a vector field $\beta_X$ for some $\beta \in \lia$ and so  the critical  points of the Morse-Bott function $\mup^\beta$. 

We point out that $X^A$ contains much information of the geometry of both the $A$-gradient map and the $G$-gradient map. Indeed, for any $x\in X$, $\mu_\lia(A\cdot x)$  is an open convex subset of $\mu_\lia (x)+\lia_x$ and $\overline{ \mu_\lia (A\cdot x) }=\mathrm{conv} (\mu_\lia (X^A \cap \overline{A\cdot x)})$, see \cite{atiyah,bg,hs}, where $\mathrm{conv}(\cdot) $ denotes the convex hull of $(\cdot).$ In particular $\mu_\lia (X^A)$ is a finte set and $\mathrm{conv} (\mu_\lia (X))=\mathrm{conv} (\mu_\lia (X^A))$ and so a polytope.  Moreover, if $\lia \subset \liep$ is a maximal Abelian subalgebra, then $\mathrm{conv} (\mup(X))$ is given by $K\mathrm{conv} ( \mu_\lia (X))$ \cite{bgi}.

The second main result proves that the existence of $\beta \in \lia$ such that the limit map associated with the gradient flow of $\mup^\beta$ defines a map from $X$ onto $X^A$. Hence, the set $X^A$ is the image of the gradient flow of the Morse-Bott function $\mup^\beta$ for some $\beta \in \lia$.

Let $\alpha_1,\ldots,\alpha_n\in \lia$ be a basis of $\lia$. Then  $\phi_\infty^{\alpha_n}\circ \cdots \circ \phi_\infty^{\alpha_1}$ defines a map from  the manifold $X$ onto   $X^{\alpha_1} \cap \cdots \cap X^{\alpha_n}=X^A$.
\begin{theorem}
Let $\alpha_1,\ldots,\alpha_n\in \lia$ be a basis of $\lia$. There exists $\delta>0$ such that for any $0<\epsilon_2,\ldots,\epsilon_n  < \delta$ we have
$
\phi_\infty^{\alpha_1+\epsilon_2 \alpha_2 +\cdots + \epsilon_n \alpha_n}=\phi_\infty^{\alpha_n}\circ \cdots \circ \phi_\infty^{\alpha_1}.
$
\end{theorem}
\section{Proof of the main results} \label{ghigi}
Suppose $X\subset Z$ is a $G$-invariant compact connected real submanifold of $Z$ with the gradient map $\mu_\mathfrak{p} : X\to \mathfrak{p}.$ If $x\in X,$  then $G_x=\{h\in G:\, g\cdot x=x\}$ denotes the stabilizer of $G$ at $x$.  If $G_x$ acts on a manifold $S$, then  $G \times^{G_x}S$ denotes the associated bundle with principal
bundle $G \ra G/G_x$ defined as the quotient of
$G \times S$ by the $G_x$-action $h  (g, s) = (gh^{-1}, h · s)$. We recall the Slice Theorem; see \cite{heinzner-schwarz-stoetzel} for details.
\begin{theorem}\label{line}[Slice Theorem \protect{\cite[Thm. 3.1]{heinzner-schwarz-stoetzel}}, \cite{sjamaar}]
If $x \in X$ and $\mup(x) = 0$, there are a $G_x$-invariant
  decomposition $T_x X = \lieg \cd x \oplus W$, open $G_x$-invariant
  subsets $S \subset W$, $\Omega \subset X$ and a $G$-equivariant
  diffeomorphism $\Psi : G \times^{G_x}S \ra \Omega$, such that $0\in
  S, x\in \Omega$ and $\Psi ([e, 0]) =x$.
\end{theorem}
\begin{corollary} \label{slice-cor} If $x \in X$ and $\mup(x) = \beta$,
  there are a $G^\beta$-invariant decomposition $T_x X = \lieg^\beta
  \cd x \, \oplus W$, open $G^\beta$-invariant subsets $S \subset W$,
  $\Omega \subset X$ and a $G^\beta$-equivariant diffeomorphism $\Psi
  : G^\beta \times^{G_x}S \ra \Omega$, such that $0\in S, x\in \Omega$
  and $\Psi ([e, 0]) =x$.
\end{corollary}
This follows applying the previous theorem to the action of $G^\beta$ on $X$. Indeed, it is well known that $G^\beta=K^\beta\exp(\liep^\beta)$ is compatible \cite{borel-ji-libro} and the orthogonal projection of $\textbf{i} \mu$ onto $\liep^\beta$ is the $G^\beta$-gradient map $\mu_{\liep^\beta}$. The group $G^\beta$ is also compatible with the Cartan decomposition of $(U^\C )^{\beta}=(U^\C )^{\textbf{i} \beta}=(U^{\textbf{i}\beta} )^\C$ and $\textbf{i} \beta$ is fixed by the $U^{\textbf{i}\beta}$-action on $\liu^{\textbf{i}\beta}$. This implies that $\widehat{\mu_{\liu^{\textbf{i} \beta}}}:Z \lra \liu^{\textbf{i} \beta}$ is
given by $\widehat{\mu_{\liu^{\textbf{i}\beta}}(z)}=\pi_{\liu^{\textbf{i}\beta}} \circ \mu+\textbf{i}\beta$, where $\pi_{\liu^{\textbf{i}\beta}}$ is the orthogonal projection of $\liu$ onto  $\liu^{\textbf{i}\beta}$, is the $U^{\textbf{i}\beta}$-shifted momentum map. The associated $G^{\beta}$-gradient map is given by $\widehat{\mu_{\liep^\beta}} := \mu_{\liep^\beta} -
\beta$. Hence, if $G$ is commutative, then we have a Slice Theorem for $G$ at every point of $X$,
see \cite[p.$169$]{heinzner-schwarz-stoetzel} and \cite{sjamaar} for more details.

If $\beta \in \liep$, then $\beta_X$ is a vector field on
$X$, i.e. a section of the bundle $TX$. For $x\in X$, the differential is a map
$T_x X \ra T_{\beta_X (x)}(TX)$. If $\beta_X (x) =0$, there is a canonical splitting $T_{\beta_X (x)}(TX) = T_x X \oplus
T_x X$. Accordingly the differential of $\beta_X$, regarded as a section of $TX$, splits into a horizontal and a vertical part. The horizontal part is the identity map. We denote the
vertical part by $\mathrm d \beta_X  (x)$. The linear map $\mathrm d \beta_X  (x)\in \End(T_x X)$ is indeed the so-called intrinsic differential of $\beta_X$, regarded as a section in the tangent bundle $TX$, at the vanishing point $x$. Let
$\{\phi_t=\exp(t\beta)\} $ be the flow of $\beta_X$.  There
  is a corresponding flow on $TX$. Since $\phi_t(x)=x$, the flow on
  $TX$ preserves $T_x X$ and there it is given by $d\phi_t(x) \in
  \Gl(T_x X)$.  Thus we get a linear $\R$-action on $T_x X$ with
  infinitesimal generator $d\beta_X  (x) $.
\begin{corollary}\label{slice-cor-2}
  If $\beta \in \liep $ and $x \in X$ is a critical point of $\mupb$,
  then there are open invariant neighborhoods $S \subset T_x X$ and
  $\Omega \subset X$ and an $\R$-equivariant diffeomorphism $\Psi : S
  \ra \Omega$, such that $0\in S, x\in \Omega$, $\Psi ( 0) =x$. (Here
  $t\in \R$ acts as $d\phi_t(x)$ on $S$ and as $\phi_t$ on $\Omega$.)
\end{corollary}
\begin{proof}
 Since $\exp:\liep \lra G$ is a diffeomorphism onto the image, the subgroup $H:=\exp(\R \beta)$ is closed and so it is compatible.  Hence, it is enough to apply the previous corollary to the $H$-action on $X$ and the value at  $x$ of the corresponding gradient map.
\end{proof}
\begin{lemma}\label{compatible}
Let $\lia\subset \liep$ be an Abelian subalgebra and let $A=\exp(\lia)$. If $x\in X$, then $A_x$ is compatible, i.e., $A_x=\exp(\lia_x)$
\end{lemma}
\begin{proof}
If $a\in A_x$, then $a=\exp(\beta)$ for a $\beta \in \lia$. Let
$
f(t)=\langle \mu_\lia (\exp(t\beta)x), \beta \rangle
$.  Then $f(1)= \langle \mu_\lia (\exp(\beta)x), \beta \rangle = \langle \mu_\lia (ax), \beta \rangle = \langle \mu_\lia (x), \beta \rangle = f(0)$ and $f'(t)=\parallel \beta_X (\exp(t\beta)x)\parallel^2\geq0 $. This implies $\beta_X (x)=0$ and so $\beta \in \liea_x$, proving  $A_x = \exp(\liea_x)$.
\end{proof}
Let  $\alpha,\beta \in \liep$ be such that $[\alpha,\beta]=0$ and let $\lia$ be the vector space in $\liep$ generated by $\alpha$ and $\beta$. By the above Lemma,  it follows that
$
X^A=X^\beta \cap X^\beta,$ where $A=\exp(\lia).$
\begin{lemma}\label{linearization}
Let $\beta, \alpha\in \liep$ be such that $[\beta,\alpha] = 0.$ There exists $\delta > 0$ such that for any $\epsilon\in (0, \delta)$ $X^{\beta + \epsilon\alpha} = X^\beta \cap X^\alpha.$
\end{lemma}
\begin{proof}
Let $\epsilon >0$ and let  $A = \exp(\mathfrak{a}),$ where $\lia=\mathrm{span}(\alpha,\beta)$.  Since the exponential map is a diffeomorphism restricted on $\liep$, it follows that $A$ is a closed and compatible subgroup of $G$. Let $X^A$ denote the fixed point set of $A$, i.e., $X^A=\{z\in X:\, A\cdot x=x\}$. By  Lemma \ref{compatible}, $X^A=X^\beta \cap X^\alpha$. By Corollary \ref{slice-cor-2},  both $X^\beta \cap X^\alpha$ and $X^{\alpha+\epsilon \beta}$ are compact submanifolds satisfying $X^\beta \cap X^\alpha \subseteq X^{\alpha+\epsilon \beta}$. Since $X^{\alpha+\epsilon \beta}$ is $A$-invariant, and so there exists  $A$-gradient map \cite{heinzner-schwarz-stoetzel}, any connected component of  $X^{\alpha+\epsilon \beta}$ contains a connected component of $X^\alpha \cap X^\beta$.

Let $x\in X^\alpha \cap X^\beta$. Let $C$ be the connected component of $x$ and let $C'$ be the connected component of $X^{\alpha+\epsilon \beta}$ containing $C$.
Since $x$ is fixed by $A$,  by the linearization theorem, Corollary \ref{slice-cor-2}, there exists $A$-invariant open subsets $\Omega \subset X$ and $S\subset  T_x X$ and a $A$-equivariant diffeomorphism $\phi : S \to \Omega$ such that $0\in S,$ $x\in \Omega$, $\phi(0) = x,$ $d\phi_0 = id_{T_x X}.$ Thus we may assume that $\Omega=\R^n$, $\alpha,\beta$ are symmetric matrices of order $n$ satisfying $[\alpha,\beta] = 0$. Moreover, $T_x  X^{\alpha+\epsilon \beta}=\mathrm{Ker}\, (\alpha+\epsilon \beta)$ and
 $T_x  X^{\alpha}\cap T_x X^{\beta}=\mathrm{Ker}\, \alpha \cap \mathrm{Ker}\, \beta$.

The matrices $\alpha$ and $\beta$ are simultaneously diagonalizable. Let $\{e_1,\ldots,e_n\}$ be a basis of $\R^n$ such that $\alpha e_i =a_i e_i$ and $\beta e_i =b_i e_i$ for
$i=1,\ldots, n$. Let $J=\{1\leq i \leq n:\, a_i b_i \neq 0\}$. Pick $\delta =\mathrm{min}\{\frac{| a_i |}{| b_i |}:\, i\in J\}$. Now,
$(\alpha +\epsilon \beta)e_i=0$ if and only if $a_i+\epsilon b_i=0$. If $a_i \neq 0$, then $b_i \neq 0$ and vice-versa. Therefore, for any $\epsilon <\delta$, we get
$
(\alpha +\epsilon \beta)e_i=0,
$
if and only if $a_i=b_i=0$. Therefore, $\mathrm{Ker}\, (\alpha+\epsilon \beta)=\mathrm{Ker}\, \alpha \cap \mathrm{Ker}\, \beta$. Since $C\subset C'$ and $T_x C=T_x C'$, keeping in mind that both $C$ and $C'$ are compact, it follows that $C=C'$. Since $X^{\alpha+\epsilon \beta}$ has finitely many connected components, it follows that there exists $\delta >0$ such that for any $0 < \epsilon < \delta$, we have
\[
X^\alpha \cap X^\beta =X^{\alpha+\epsilon \beta},
\]
concluding the proof.
\end{proof}
\begin{theorem}\label{main1}
Let $\lia \subset \liep$ be an Abelian subalgebra and let $A=\exp(\lia)$. Then the set $$\left\{ \alpha \in \lia:\, X^A=X^\alpha\right\}$$ is dense.
\end{theorem}
\begin{proof}
Let $\alpha_1,\ldots,\alpha_n$ be a basis of $\lia$. Then
\[
X^A=X^{\alpha_1} \cap \cdots \cap X^{\alpha_n}.
\]
By the above Lemma, there exists $\delta >0$ such that for any $\epsilon_2,\ldots,\epsilon_n < \delta$, we have
\begin{equation}\label{oi}
X^A =X^{\alpha_1+ \epsilon_2 \alpha_2+\cdots +\epsilon_n \alpha_n}
\end{equation}
Let $\alpha \in \lia$ different form $0$. It is well known that there exists $\alpha_2,\ldots \alpha_n\in \lia$ such that $\alpha,\alpha_2,\ldots,\alpha_n$ is a basis of $\lia$. By $(\ref{oi})$, for any neighborhood  $U$ of $\alpha$,  there exists $\beta \in U$ such that
$
X^A =X^\beta,
$
concluding the proof.
\end{proof}
The following Lemma is proved in \cite{bw}, see also  \cite[pag. 1036]{BT}.
\begin{lemma}\label{linearization2}
Let $x\in X$ and $\beta, \alpha \in \liep$ be such that $[\beta, \alpha]= 0.$ Set $y:= \lim_{t\to \infty}\exp(t\beta)x$ and $z:= \lim_{t\to \infty}\exp(t\alpha)y.$ Let $\delta$ be as in Lemma \ref{linearization}. Then for $0<\epsilon<\delta,$
$$
\lim_{t\to \infty}\exp(t(\beta + \epsilon\alpha)) = z.
$$
\end{lemma}
As a consequence of the above Lemma we get the following result.
\begin{theorem}
Let $\alpha_1,\ldots,\alpha_n$ be a basis of $\lia$. Let $x\in X$. Set $x_1:= \lim_{t\to \infty}\exp(t\alpha_1) x$ and $x_{i}= \lim_{t\to \infty}\exp(t\alpha_i) x_{i-1}$ for $i=2,\ldots,n$. Then there exists $\delta>0$ such that for $0<\epsilon_2,\ldots,\epsilon_n<\delta,$ we have
$$
\lim_{t\to \infty}\exp(t(\alpha_1 + \epsilon_2 \alpha_2 +\cdots + \epsilon_n \alpha_n))x = x_n,
$$
for any $x\in X$. In particular, $
\phi_\infty^{\alpha_1+\epsilon_2 \alpha_2 +\cdots + \epsilon_n \alpha_n}=\phi_\infty^{\alpha_n}\circ \cdots \circ \phi_\infty^{\alpha_1}.
$
\end{theorem}
\begin{proof}
By Theorem \ref{main1}, there exists $\delta >0$ such that for any $0<\epsilon_2,\ldots,\epsilon_n<\delta$, we have
\[
X^A=X^{\alpha_1+\epsilon_2 \alpha_2+\cdots+\epsilon \alpha_n}. 
\]
Let $A=\exp(\lia)$. Let $z\in X^A$. By Corollary \ref{slice-cor-2}, there exists $A$-invariant open subsets $\Omega \subset X$ and $S\subset T_z X$ and a $A$-equivariant diffeomorphism $\phi : S \to \Omega$ such that $0\in S,$ $z\in \Omega$, $\phi(0) = z,$ $d\phi_0 = id_{T_z X}.$  Let $x\in X$. Set $x_1:= \lim_{t\to \infty}\exp(t\alpha_1) x$ and $x_{i}= \lim_{t\to \infty}\exp(t\alpha_i) x_{i-1}$ for $i=2,\ldots,n$. If $x_n \in \Omega$, we may choce $\delta >0$ such that for any $0<\epsilon_2,\ldots,\epsilon_n <\delta$, we have
\[
\lim_{t\mapsto} \exp(t(\alpha_1 + \epsilon_2\alpha_2+\cdots+\epsilon \alpha_n))x=x_n.
\]
By compactness of $X^A$ there exist open subsets $\Omega_1,\ldots,\Omega_k$  satisfying the above property and such that
\[
X^A \subseteq \Omega_1 \cup \cdots \cup \Omega_k.
\]
Let $x\in X$.  Set $x_1:= \lim_{t\to \infty}\exp(t\alpha_1) x$ and $x_{i}= \lim_{t\to \infty}\exp(t\alpha_i ) x_{i-1}$ for $i=2,\ldots,n$. If $x_n \in \Omega_j$, for some $j=1,\ldots,k$, then there exits $\delta_j>0$ such that any $0<\epsilon_2,\ldots,\epsilon_n<\delta_j$ we have
\[
\lim_{t\mapsto +\infty} \exp(t(\alpha_1 +\epsilon_2 \alpha_2+\cdots+\epsilon \alpha_n))x=x_n.
\]
Let $\delta=\mathrm{min} \{\delta_1,\ldots,\delta_k\}$. Then for any $0<\epsilon_2,\ldots,\epsilon_n<\delta$ we have
\[
\lim_{t\mapsto+\infty} \exp(t(\alpha_1 +\epsilon_2 \alpha_2 +\cdots+\epsilon_n \alpha_n))x=x_n,
\]
for any $x\in X$, concluding the proof.
\end{proof}


\begin{thebibliography}{15}
%
\bibitem{ad}
Adams, J.F., \textit{Lectures on Lie groups},
W. A. Benjamin, Inc., New York-Amsterdam 1969 xii+182 pp.
%
\bibitem{atiyah}
Atiyah, M.F.,  \textit{Convexity and commuting Hamiltonians}, Bull. London Math. Soc., \textbf{14 (1)}, (1982), 1--15.
%
\bibitem{LA} Biliotti L., Ghigi A. and Heinzner P.,
\textit{Polar orbitopes}, Comm. Ann. Geom. \textbf{21 (3)}, (2013), 1--28.
%
\bibitem{bgi}
Biliotti, L.,  Ghigi, A., Heinzner, P., 
\textit{Invariant Convex sets in Polar Representations}, Israel J. Math \textbf{213},
(2016), 423--441.
\bibitem{bg}
Biliotti L., Ghigi, A., \textit{Remarks on the Abelian Convexity Theorem}, Proc. Amer. Math. Soc. \textbf{146 (12)}, (2018),
5409--5419.
%
\bibitem{bw}
Biliotti L., Windare, O.J.,
\textit{Stability, analytic stability for real reductive Lie groups}, J. Geom. Anal. \textbf{33}, (2023)  (31 pages)
%
\bibitem{borel-ji-libro} Borel A., ~Ji, L., \textit{Compactifications of symmetric and locally symmetric spaces}. Mathematics: Theory \& Applications. Birkh\"auser Boston Inc., Boston, MA., (2006).
%
\bibitem{Bott} Bott, R, \textit{Non-degenerate critical manifolds},  Ann. of Math. \textbf{60}, (1954)  248--261.
%
\bibitem{BT} Bruasse, A., Teleman, A,
\textit{Harder-Narasimhan and optimal destabilizing vectors in complex geometry}, Ann. Inst. Fourier (Grenoble) \textbf{55 (3)}, (2005), 1017–-1053.
%
\bibitem{DK}
{\sc Duistermaat, J.J., Kolk, J. A. C.},
\newblock {\emph Lie groups},
\newblock  Universitext. Springer-Verlag, Berlin, 2000.
%
\bibitem{PG} ~Heinzner, P., Schwarz, G. W., \textit{Cartan decomposition of the moment map}, Math. Ann. 337, (2007), 197-232.
%
\bibitem{heinzner-schwarz-stoetzel} ~Heinzner, P., Schwarz, G.~W., St{\"o}tzel, H., \textit{Stratifications with respect to actions of real reductive groups.} Compos. Math., 144(1), (2008), 163--185.
%
\bibitem{hs}
Heinzner, P.,  Sch\"utzdeller, P.,  
Convexity properties of gradient maps. Adv. Math., \textbf{225(3)}, (2010), 1119--1133.
%
\bibitem{Kirwan} Kirwan F.,
\textit{Cohomology of quotients in symplectic and algebraic Geometry}, Math. Notes \textbf{31}, Princeton, (1984).
%
\bibitem{sjamaar}
Sjamaar, R.
\textit{Convexity properties of the momentum mapping re-examinated.} Adv. Math. \textbf{138 (1)}, (1998), 46-91.	
%	
%\bibitem{Teleman} Teleman A., \textit{Symplectic Stability, Analytic Stability in Non-Algebraic Complex Geometry}, Int. J. Math. \textbf{15}(2), (2004), 183-209.
\end{thebibliography}
\end{document}